\documentclass{amsart}

\usepackage{ifpdf}
\usepackage{amsmath, amsthm, amsfonts, amssymb, amscd}

\usepackage{hyperref}

\newcommand{\Z}{{\ensuremath{\mathbb{Z}}}}
\newcommand{\R}{{\ensuremath{\mathbb{R}}}}
\newcommand{\N}{{\ensuremath{\mathbb{N}}}}
\newcommand{\Q}{{\ensuremath{\mathbb{Q}}}}
\newcommand{\T}{{\ensuremath{\mathbb{T}}}}
\newcommand{\eL}{{\ensuremath{\mathbb{L}}}}

\newcommand{\eM}{{\ensuremath{\mathcal{M}}}}
\newcommand{\B}{{\ensuremath{\mathcal{B}}}}

\renewcommand{\hom}{\mathrm{Homeo}}

\newcommand{\de}{\textnormal{Desv}}

\newcommand{\dm}{\textnormal{diam}}
\newcommand{\cv}{\textnormal{Conv}}

\renewcommand{\Pr}{\textnormal{Pr}}

\def \ti{\tilde}

\newtheorem{main}{Theorem}[section]

\newtheorem{defi}{Definition}[section]
\newtheorem{teo}[defi]{Theorem}
\newtheorem{pro}[defi]{Proposition}
\newtheorem{cor}[defi]{Corollary}
\newtheorem{lem}[defi]{Lemma}

\newtheorem{afp}{Claim}
\newtheorem{ob}[defi]{Remark}

\title[Rotation sets for commuting torus homeomorphisms]{Restrictions on rotation sets for commuting torus homeomorphisms}
\author[D. M. S. Castelblanco]{Deissy M. S. Castelblanco}


\email{deissy@ime.usp.br}

\begin{document}

\begin{abstract}
Let $K_1,\: K_2\subset \R^2$ be two convex, compact sets. We would like to know if there are commuting torus homeomorphisms $f$ and $h$ homotopic to the identity, with lifts $\tilde f$ and $\tilde h$, such that $K_1$ and $K_2$ are their rotation sets, respectively. In this work, we proof some cases where it cannot happen, assuming some restrictions on rotation sets. 
\end{abstract}

\maketitle


\section{Introduction}

In a landmark paper \cite{mz89}, Miziurewicz and Zieman have introduced the concept of rotation sets for torus homeomorphisms in the isotopy class of the identity. This concept generalizes the notion of rotation number of circle homeomorphisms defined by H. Poincar\'e, and  has proven to be a very effective tool in describing behaviour of toral dynamics. For instance, the analysis of the rotation sets can supply information on the abundance of periodic points \cite{f89}, the topological entropy of the map \cite{l} and sensitive dependence on initial conditions \cite{kf13}. We define rotation sets properly on section \ref{rotset} but, in a nutshell, the rotation set of a torus homeomorphisms is the convex closure of all rotation vectors for individual points, that is, 
$$\rho(\tilde f, \tilde x)= \lim_{n\to\infty}\frac{\tilde f^n(\tilde x)- \tilde x}{n},$$
when such limit exists. \medskip

This study of the rotation theory, while it has seen significant advances in the last decade, still has left several relevant open questions, and the most relevant is probably to characterize the possible rotation sets for torus homeomorphisms.  Up until very recently, it was not known if there existed a convex compact set that was not the rotation set of some torus homeomorphism, but the first example of such sets appeared in \cite{ct}. The standing conjecture on the subject, proposed by Franks and Misiurewicz on \cite{fm}, poses that the following convex sets could not be the rotation sets of torus homeomorphisms.
	\begin{itemize}
		\item[i] A segment of a line with irrational slope that contains a rational point in its interior.
		\item[ii] A segment of a line with irrational slope without rational points.
		\item[iii] A segment of a line with rational slope without rational points.
	\end{itemize}
	
In \cite{ct}, the authors showed that the conjecture is true in case i, but A. Avila has announced a construction for a torus homeomorphism whose rotation set lies in the class ii above. It is still not know if the unit circle can be realized as rotation set. A very recent and relevant work on the subject, \cite{pat} exhibits a one parameter family of torus homeomorphisms and study the corresponding family of rotation sets.\medskip

In this work we attempt to analyse possible connections between the study of group actions and the rotation theory for torus homeomorphisms. In particular, we were interested in determining, given a pair of commuting torus homeomorphisms, if there was some relationship between the respective rotation sets. The first question we tackled was to see if given compact convex sets $K_1$ and $K_2$ such that each $K_i$ was a rotation set for a torus homeomorphism, could we also find a pair of commuting homeomorphisms $f_i$ such that $K_i$ was their respective rotation set?  Our first result shows that this is not true in general. Specifically, we show that:

\begin{main}\label{ncom0}
If $\rho(\tilde f)$ is a segment $J$ with rational slope containing rational points and $\rho(\tilde h)$ is whatever following cases:
\begin{enumerate}
	\item It has nonempty interior.
	\item It is any nontrivial segment nonparallel to $J$.
\end{enumerate}
Then, $\tilde f$ and $\tilde h$ do not commute.
\end{main}

The techniques developed for Theorem \ref{ncom0} actually gives us a stronger result. It concerns the rigidity of torus homeomorphisms as described by their deviations that is, 

\begin{equation*}
\de (\tilde f)= \sup_{\stackrel{n\in \Z}{x,y\in \R^2}} 
																		\left\|\left(\tilde f^{n}(x)-x\right)-\left(\tilde f^{n}(y)-y\right)\right\|.
\end{equation*}
\medskip

If $v\in\R^2_*$, we define \emph{deviation of $\ti f$ in the direction $v$}, by
\begin{equation*}
\de_v (\ti f)= \sup_{\stackrel{n\in \Z}{x,y\in \R^2}} 
\left|\Pr_v\left[\left(\ti f^{n}(x)-x\right)-\left(\ti f^{n}(y)-y\right)\right]\right|.
\end{equation*}
\medskip

If deviation of $\ti f$ is unbounded, we will write $\de_v (\ti f)=\infty$.
\medskip

We said $f$ is \emph{annular}, if some lift $\tilde f:\R^2\to \R^2$ of $f$ has uniformly bounded displacement in a rational direction; that is, there are 
$v\in \Z_*^2$ and $M>0$ such that 
\[ \left|\left\langle \tilde f^n(x)-x,v\right\rangle\right|\leqslant M,\] 
for every $x\in\R^2$ and $n\in\Z$. In that case we said $\ti f$ is $v^\bot$-\emph{annular}. 
\medskip

The technical result shows that in some cases, if $f$ and $g$ are commuting homeomorphisms, then if $f$ has deviation restrictions $g$ must also has deviation restrictions.

\begin{main}\label{d0}
Let $f,  h\in \hom_0(\T^2)$ commute and let $\tilde f,\tilde h$ be respective lifts.
If $\tilde f$ is $v$-annular, for some $v\in\Z^2_*$, then $\de_{v^{\bot}}(\tilde h)<\infty$ or $\de (\tilde f)<\infty$.
\end{main}
\medskip

A previous version of Theorem \ref{d0} for area preserving homeomorphisms was obtained by Benayon in his doctoral thesis \cite{mau}.\medskip

The paper is organized as follows: The next section formally introduces all the relevant objects and describes the necessary tools, and section \ref{dem} is devoted to the proofs of the main results.


\section{Preliminaries}\label{pre}

\subsection{Notations}

We will denote $\N_*$, $\Z_*$ and $\R_*$ the naturals, integers and reals numbers without zero, respectively.  
We denote the two-dimensional torus $\T^2=\R^2 / \Z^2$. Let $\hom_0(\T^2)$ be the space of homeomorphisms homotopic to the identity of 
$\T^2$, $f\in \hom_0(\T^2)$, $\tilde{f}:\R^{2}\to \R^{2}$ be any lift of $f$ and $\pi:\T^2\to \R^2$ the canonical projection.\medskip

We also denote by $\left\langle\: ,\right\rangle$ the usual scalar product in $\R^2$, $\Pr_v:\R^2\to\R$ the orthogonal projection given by $$\Pr_v(x)=\frac{\left\langle x ,\: v\right\rangle}{\left\|v\right\|}.$$
We shall denote in $\R^2$ the integer translations $T_1(z)=z+(1,0)$ and $T_2(z)=z+(0,1)$, and the projections $\Pr_1(x,y)=x$ and $\Pr_2(x,y)=y$.
If $v=(a,b)\in\R^2$, we write $v^{\bot}=(-b,a)$.\medskip

We write $\cv (A)$ for the hull convex of $A$. 


\subsection{Atkinson's Lemma}

Another result that will be useful in the next section is Atkinson's Lemma on ergodic theory. See \cite{a}.

\begin{pro}
Let $(X,\B,\mu)$ be a probability space and let $T:X\to X$ be an ergodic automorphism. Consider $\phi:X\to \R$ an integrable map such that 
$\int \phi \: d\mu=0$, then for every $B\in \B$ and every $\epsilon >0$, 
\begin{align*}
	\mu\left(\bigcup_{n\in\N}B\cap T^{-n}(B)\cap\left\{x\in X:\; \left|\sum_{i=0}^{n-1}\phi(T^i(x))\right|<\epsilon\right\}\right)=\mu(B).
\end{align*}
\end{pro}
\medskip

\begin{cor}\label{atk}
Let $X$ be a separable metric space, $f:X\to X$ be a homeomorphism and $\mu$ an $f$-invariant ergodic nonatomic Borel probability measure. If 
$\phi\in \eL_1(\mu)$ is such that $\int \phi \: d\mu=0$, then for $\mu$-a.e. $x\in X$, there is an increasing sequence $(n_i)_{i\in\N}$ of integers such that
\begin{equation*}
  f^{n_i}(x)\to x\ \ \ \textit{and} \ \ \
  \sum_{k=0}^{n_i-1}(\phi_1\circ f^{k})(x)\to 0, \ \ \ \textit{as} \ \ \ i\to \infty.
\end{equation*} 
\end{cor}
\medskip


\subsection{Rotation theory}
\label{rotset}

From now on, let $f\in \hom_0(\T^2)$ and let $\tilde{f}:\R^{2}\to \R^{2}$ be any lift of $f$.\medskip

The following definitions were introduced by Misiurewics and Ziemian in \cite{mz89}.

\begin{defi}
Let $x\in\T^2$. If
\begin{equation}\label{rot}
 \lim_{n\to\infty}\frac{\tilde{f}^n(\tilde{x})-\tilde{x}}{n}
\end{equation}
exist, for some $\tilde{x}\in\pi^{-1}(x)$, then the limit (1) is denoted by $\rho(\tilde{f},x)$ and call 
\textsl{the rotation vector of $x$ by $\tilde{f}$}. 
\end{defi}
\medskip

\begin{defi}
The point-wise rotation set of $\tilde{f}$ is:
 \[\rho_p(\tilde{f})=\bigcup_{x\in \T^2}\rho(\tilde{f},x)\]
\end{defi}

\begin{defi}
The \textsl{rotation set of $\tilde{f}$}, denoted by $\rho(\tilde{f})$, is the set of points $v\in\R^2$, such that there exist sequences
 $(n_i)\in \Z$ and $(x_i)\in \R^2$ with $n_i\to \infty$ as $i\to\infty$, and 
\[\lim_{i\to \infty} \frac{\tilde{f}^{n_i}(x_i)-x_i}{n_i}=v.\]
\end{defi}

Let $\phi:\T^2\to \R^2$ be the displacement function defined by $\phi(x)=\tilde f(\tilde x)-\tilde x$, where $\tilde x\in \pi^{-1}(x)$ and 
$\pi:\R^2\to \T^2$ is the natural projection.\medskip

Denote the space of all  $f$-invariant probability measures on $\T^2$ by $\eM (f)$ and the subset of ergodic measures by $\textit{M} _e(f)$.

\begin{defi}
If $\mu\in \eM (f)$, define its rotation vector by $\rho(\mu,f)=\int \phi\ d\mu$. 
Also the sets
\begin{align*}
	\rho_m(\ti f)=&\left\{\rho(\mu,\ti f); \mu\in \eM (f) \right\}\ and \\
	\rho_e(\ti f)=&\left\{\rho(\mu,\ti f); \mu\in \eM_e (f)\right\}.
\end{align*}
\end{defi}
\medskip

\begin{pro}\label{birkof}
If $\mu\in \eM_e(f)$, then $\rho(\mu, \ti f)=\rho(\tilde f, x)$, for $\mu$-a.e. $x\in\T^2$.
\end{pro}

\begin{proof}
	If $\mu\in M_e(f)$, by the ergodic theorem follows that $\mu$-a.e. $x\in T^2$,
	
	\[\lim_{n\to \infty}\frac{1}{n}\sum_{k=0}^{n-1}{\phi\circ f^k(x)}=\int \phi\ d\mu.\] 
	
	Therefore for $\mu$-a.e. $x\in\T^2$ and every $\tilde x\in \pi^{-1}(x)$, 
	\begin{equation*} 
	\rho(\tilde{f},x)=\lim_{n\to\infty}\frac{\tilde{f}^n(\tilde{x})-\tilde{x}}{n}=\rho(\mu,f).
	\end{equation*}
\end{proof}
\medskip

\begin{teo}[\cite{mz89}]\label{convexo}
It holds
	\[\rho(\tilde f)=\cv( \rho_p(\tilde f))= \cv (\rho_e(\tilde f))=\rho_m(\ti f).\]
\end{teo}
\medskip


\subsection{Bounded deviation}

In this part we study the relationship between deviation and rotation set concepts.

\begin{lem}\label{proye}
If $\de_u(\ti f)$ is bounded, for some $u\in\R^2_*$, then there is $a\in\R$ such that, $\Pr_u(\rho(\ti f))=\left\{a\right\}$. 
\end{lem}
\medskip

In that case, $\rho (\tilde f)\subset L(u^{\bot},a):=\left\{tu^{\bot}+au;\;t\in\R\right\}$.\medskip

\begin{proof}
From definition there is $M>0$, such that for all $x, y\in\R^2$ and $n\in\Z$,
\begin{equation}\label{desv1}
		\left|\Pr_u\left[\left(\tilde f^{n}(x)-x\right)-\left(\tilde f^{n}(y)-y\right)\right]\right|\leqslant M. 
\end{equation}

Set $p=\rho(\ti f, x_0)$ for some $x_0\in\T^2$, which exist since $\eM_e(f)$ is not empty, with $\Pr_u(p)=a\in\R$. By $(\ref{desv1})$ follows that 
$$\displaystyle \lim_{n\to\infty}\Pr_u\frac{\tilde f^{n}(y)-y}{n}=a,$$ for every $y\in\R^2$. Therefore 
$\left\{a\right\}=\Pr_u(\rho_p(\ti f))=\Pr_u(\rho(\ti f))$. 
\end{proof}
\medskip

\begin{ob}\label{int}
It is an elementary exercise to show that if $\tilde f$ is $v$-annular, for some $v\in\Z_*^2$, then 
$\rho (\tilde f)\subset L(v,a):=\left\{tv+av^{\bot};\;t\in\R\right\}$.
\end{ob}
\medskip

The following result shows that the converse to Remark \ref{int} is true in several case.

\begin{teo}[\cite{pa}]\label{pa}
If $\rho(\tilde f)$ is a segment with rational slope containing rational points, then $f^k$ is annular, for some $k\in\N$.
\end{teo}
\medskip

\begin{pro}\label{wanul}
If $\rho(\ti h)$ is a nontrivial segment, contained in the straight line $$L(w,b)=\left\{tw+bw^{\bot};\;t\in\R\right\},$$ where 
$b=\Pr_{w^{\bot}}(\rho(\ti h))$ and $w\in\R^2$, then $\de_{v^{\bot}}(\ti h)$ is unbounded, for every $v\in\R^2$ nonparallel to $w$.
\end{pro}

\begin{proof}
Suppose by contradiction that there is $M>0$ such that, $\de_{v^{\bot}}(\ti h)\leqslant M$. We know because of the Lemma \ref{proye}, that  
$\rho(\ti h)$ is a subset of $$L(v,a)=\left\{tv+av^{\bot};\;t\in\R\right\},$$ where $a=\Pr_{v^{\bot}}(p)$, for every 
$p\in\rho(\ti h)$. So $\rho(\ti h)\subset [L(v,a)\cap L(w,b)]$. As $v$ and $w$ are not parallels, we can deduce that	
$\rho(\ti h)=\left\{p\right\}$. Absurd! 
\end{proof}
\medskip


\subsection{Annularity and commuting homeomorphisms}

Let $f, h\in \hom_0(\T^2)$ commute. In \cite{pk} was shown there exist lifts of $f$ and $h$ that also commute. Let $\ti f,\ti h$ be respective commuting lifts.
\medskip

The next lemma is contained in (\cite{mau}, Proposition 3.1).  We include the proof here for the sake of completeness.

\begin{lem}\label{mau}
If $\ti f$ is $(0,1)$-annular and $\de_{(0,1)^{\bot}}(\ti h)$ is unbounded then there exist a set $E$, such that

\begin{itemize}
\item[i.)] The set $E$ is open, $\ti f$-invariant and $T_2(E)=E$.
\item[ii.)] The connected components of $E$ have uniformly bounded diameter. 
\item[iii.)] For every $z\in \R^2$, there is $u\in \Z^2$ such that $(z+u)\in E$, that is, $\pi(E)=\T^2$.
\end{itemize}

\end{lem}

\begin{proof}
$i.)$ Denote by   \[H=\left\{(x,y)\in\R^2;\: x\geqslant 0\right\}.\] 
the closed half plane. And define the set  
$$B=\overline{\bigcup_{n\in\Z}\tilde f^n(H)}.$$

Then, $\tilde f(B)=B$ and $T_2(B)=B$. Since $\tilde f$ is $(0,1)$-\emph{annular} there is $M>0$ such that for every $z\in\R^2$,
\begin{displaymath}
\left|\Pr_1[\tilde f^n(z)-z]\right|\leqslant M.
\end{displaymath}

Thus, $$B\subset\left\{(x,y)\in\R^2;\; x\geqslant -M \right\}.$$ 

Denote by $\Delta=\sup_{z\in\R^2}\left\|\tilde f(z)-z\right\|$ and $A^\prime=B\setminus\left[T^k_1(int(B))\right],$ where $k>2M+\Delta +1$. 
Note that $\ti f(A^\prime)=A^\prime$.\medskip

Hence the set
$$A^{\prime\prime}=\left\{(x,y)\in\R^2;\: M\leqslant x\leqslant M + \Delta + 1\right\}\subset A^\prime.$$

Let $A$ be a connected component of $A^\prime$, such that $A^{\prime\prime}\subset A$. We can see that $\tilde f(A)\cap A\neq \emptyset$, thus 
$\tilde f(A)=A$. Furthermore $T_2(A)=A$, $\dm\:(\Pr_1(A))\leqslant k+2M$ and $\: (T_2\!\circ\! \tilde h^n)(A)=\tilde h^n(A)$, for every $n\in\Z$. 
\medskip

Consider $l\in\N$ with $l>2(k+2M+1)$ such that the set,
$$F=\left\{(x,y)\in\R^2;\; k+M \leqslant x \leqslant l-M\right\},$$ 
satisfies $\dm\:(\Pr_1(F))>k+2M+2$.\medskip 

We can deduce, 
	\[\dm\:[\Pr_1(B\setminus int(T_1^{l+k}(B)))]\leqslant l+k+2M. \] 

Therefore, $$\dm\:[\Pr_1(B\setminus int(T_1^{2l+k}(B)))]\leqslant 2(l+k+2M).$$

Since, $\de_{(1,0)}(\tilde h)$ is unbounded, there are $z_1,\: z_2\in\R^2$ and $n\in\Z$ such that,
\[\left|\Pr_1[(\tilde h^n(z_1)-z_1)-(\tilde h^n(z_2)-z_2)]\right|>2(l+k+2M).\] 

As $\dm\:(\Pr_1(A))\leqslant k+2M$, we have in particular that if $z_1,\: z_2\in A$ then, 
$$\left|\Pr_1[\tilde h^n(z_1)-\tilde h^n(z_2)]\right|>2l+k+2M.$$

So we may suppose $\tilde h^n(z_1)\in A$ and $\tilde h^n(z_2)\in T_1^l(A)$, replacing $\ti h^n$ by $\ti h^n+T^j_1$, for some $j\in\Z$, if necessary. \medskip

We define, $$E=[T^k_1(int(B))\setminus T_1^l(B)]\setminus \tilde h^n(A).$$ We can see that $E$ is open, $\tilde f$-invariant and 
$T_2(E)=E$, because $B$ and $\tilde h^n(A)$ are.\medskip

$ii.)$ We most show that there exist $M_0>0$ such that, $\dm(U)\leqslant M_0$, being $U$ any connected component of $E$.\medskip

Since $E$ is contained in $T^k_1(int(B))\setminus T_1^l(B)$, we deduce that 
$$\dm\:(\Pr_1(E))\leqslant (l+M)-(k-M)=M_1.$$ 

Let $\gamma:[0,1]\to \R^2$ be a path contained in $\tilde h^n(A)$, from $A$ to $T^l_1(A)$. As 
$(T^j_2\! \circ\! \tilde h^n)(A)=\tilde h^n(A)$, for every $n,\: j\in\Z$, so the path $T^j_2(\gamma)$ is also contained in 
$\tilde h^n(A)$, for every $j\in\Z$, from $A$ to $T^l_1(A)$. But $\gamma$ is compact, so there exists $a,\: b\in\R$ such that 
$a\leqslant \Pr_2(z)\leqslant b$, for every $z\in \gamma$. Then, for any connected component $U$ of $E$, 
$diam(\Pr_2(U))\leqslant \left|b-a+1\right|$.\medskip

Thus, set $M_2=\left|b-a+1\right|$ and $M_0=\sqrt{M_1^2+M_2^2}$.\medskip

$iii.)$ Let us see that for every $z\in\R^2$, there exist $u\in\Z^2$ such that $z+u\in E$.\medskip

Suppose by contradiction that there exist $z\in\R^2$ such that 
$(z+\Z^2)\cap E=\emptyset$. But there is $u\in\Z^2$ such that $w=z+u\in F$, so $w\in \tilde h^n(A)$. Since 
$\dm(\Pr_1(F))>k+2M+2$, we can choose $u$ such that $T_1^j(w)\in (F\cap \tilde h^n(A))$, for $j=0,\ldots,k+2M+1$.
Then $\tilde h^{-n}(T_1^j(w))\in A$, are $k+2M+2$ translate of a point in $A$. It contradicts that $\dm(\Pr_1(A))\leqslant k+2M$.

\end{proof}
\medskip


\subsection{An auxiliary result}

Next lemma can be found in \cite{kk}.

\begin{lem}\label{kk}
Let $f\in \hom_0(\T^2)$, $A\in SL(2,\Z)$ and $h\in \hom(\T^2)$ isotopic to $f_A$. Let $\tilde f$ and $\tilde h$ be respective lifts. Then 
$$\rho(\tilde h \tilde f \tilde h^{-1})=A\rho(\tilde f).$$ In particular, $\rho(\tilde A \tilde f \tilde A^{-1})=A\rho(\tilde f).$
\end{lem}
\medskip

\begin{ob}\label{v}
If $\tilde f$ is $v$-annular, for some $v\in\Z_*^2$, then there exist some homeomorphism conjugated to $\tilde f$ that is $(0,1)$-annular. In fact, given 
$v=(q,p)\in \Z^2$, there exist integers $a,\: b$, such that $pa+qb=1$. Consider the matrix
	\[ A=
	\begin{bmatrix}
		p & -q\\
		b & a
	\end{bmatrix}.
	\]
Hence $det(A)=1$, but $A\cdot(q,p)=(0,1)$, so $Lema\; \ref{kk}$ shows that $A\tilde f A^{-1}$ is $(0,1)$-annular.
\end{ob}
\medskip


\section{Theorems A and B} \label{dem}

\subsection{Noncommuting homeomorphisms}

Let us show how Theorem \ref{ncom} follows from Theorem \ref{d1}.\medskip

\begin{main}\label{ncom}
If $\rho(\tilde f)$ is a segment $J$ with rational slope containing rational points and $\rho(\tilde h)$ is whatever following cases:
\begin{enumerate}
	\item It has nonempty interior.
	\item It is any nontrivial segment nonparallel to $J$.
\end{enumerate}
Then, $\tilde f$ and $\tilde h$ do not commute.
\end{main}

\begin{proof}
Suppose by contradiction that $\tilde f$ and $\tilde h$ commute. By Theorem \ref{pa} there is $k\in\N$ such that $\tilde f^k$ is annular. Denote $\ti g=\ti f^k$, so $\tilde g$ and $\tilde h$ commute and suppose that $\ti g$ is $v$-\emph{annular}, for some 
$v\in \Z^2_*$.\medskip

We claim that whatever cases (1) or (2) above for $\rho(\ti h)$, $\de_{v^{\bot}}(\tilde h)$ is unbounded. In fact, the first case follows from Lemma \ref{proye} and the second holds applied Proposition \ref{wanul}.\medskip

Hence, the Theorem \ref{d1} implies $\de (\tilde g)<\infty$. Therefore, $\tilde g$ is a pseudo-rotation. It yields a contradiction because $\rho(\tilde f)$ is a nontrivial segment.
\end{proof}
\medskip


\subsection{Theorem B}
From now on let $f, h\in \hom_0(\T^2)$ commute and let $\ti f,\ti h$ be respective commuting lifts.
\medskip

In the proof of next proposition, we denote by $L=\left\{t(0,1);\; t\in\R\right\}$.\medskip

\begin{pro}\label{d}
If $\tilde f$ is $(0,1)$-annular then $\de_{(1,0)}(\tilde h)<\infty$ or 
$\de (\tilde f)<\infty$.
\end{pro}

\begin{proof}
Suppose that $\de_{(1,0)}(\tilde h)$ is unbounded. Then there is a set $E$ with the properties in Lemma \ref{mau}. \medskip

Since the connected components of $E$ have uniformly bounded diameter, we can consider $M_0>0$ such that for every connected component $U$ of 
$E$, $\dm\;(U)< M_0$.

\begin{afp}\label{m1}
There exist $\overline{n}\in \Z^+$, $\: \overline{m}\in \Z$ and $U$ a connected component of $E$, such that 
$$\left\|\tilde f^{k\overline{n}}(x)-x-k(0,\overline{m})\right\|\leqslant M_0,$$ for all $k\in\N$ and every $x\in U$.
\end{afp}

Indeed, because of the Remark \ref{int}, we may suppose that $\rho(\tilde f)\subset L$, that is,  
$\rho(\tilde f)=\left\{0\right\}\times \left[a,b\right]$, with $a\leqslant b$, for $a,b\in\R$.\medskip 

Since $(0,a)$ is an extremal point of $\rho(\tilde f)$, from Theorem \ref{convexo} there exist $\mu\in \eM_e(f)$ such that for 
$\mu$-a.e $x\in\T^2$, 
$$\rho(\mu,f)=\int \phi\ d\mu=(0,a),$$ 
where $\phi:\T^2\to \R^2$ is the displacement function defined by 
$\phi(x)=\tilde f(\tilde x)-\tilde x$ for $\tilde x\in \pi^{-1}(x)$. Then, follows from Proposition \ref{birkof}, that for $\mu$-a.e $x\in\T^2$ and every $\tilde x\in \pi^{-1}(x)$, 
$$\lim_{n\to\infty}\frac{\tilde f^n(\tilde x)-\tilde x}{n} =(0,a).$$

Consider the first projection $\Pr_1:\R^2\to \R$ and the function $\phi_1:\T^2\to\R$, defined by 
$\phi_1(x)=\Pr_1(\phi)(x).$\medskip

Hence, $$\int \phi_1\ d\mu=0.$$

By the Atkinson's Lemma (Corollary \ref{atk}), applied to the function $\phi_1$, we have that for $\mu$-a.e $x\in \T^2$, there is an increasing sequence of positive integers $(n_i)_{i\in\N}$, such that

\begin{equation*}
  f^{n_i}(x)\to x\ \ \ \textit{and} \ \ \
  \sum_{i=0}^{n_i-1}(\phi_1\circ f^{k})(x)\to 0, \ \ \ \textit{as} \ \ \ i\to \infty.
\end{equation*} 

That is, for $\mu$-a.e $x\in \T^2$ and $\ti x\in \pi^{-1}(x)$,

\begin{equation}\label{at}
   \Pr_1[\tilde f^{n_i}(\tilde x)- \tilde x]\to 0, \ \ \ \textit{as} \ \ \ i\to \infty.
\end{equation} 

Let $x\in T^2$ satisfying (\ref{at}) and $\tilde x=\pi^{-1}(x)$, by property \textit{iii.)} of $E$, there is 
$u\in\Z^2$ such that $\tilde x+u=w\in U$, for some connected component $U$ of $E$. Hence, there are sequences 
$(n_i)_{i\in\N}\in \Z^+$ and $(m_i)_{i\in\N}\in\Z$, such that $\tilde f^{n_i}(w)-T_v^{m_i}(w)\to 0$, as 
$i\to \infty$.\medskip

As $ U$ is open set, let $\epsilon>0$ be such that $B_{\epsilon}(w)\subset U$. Then there is $i_0\in\N$ such that for all $i>i_0$, 
$\tilde f^{n_i}(w)\in T_v^{m_i}(B_{\epsilon}(w))\subset T_v^{m_i}(U)$. In particular, for some $\overline{m}\in\Z$, there is 
$\overline{n}\in\Z^+$ such that $\tilde f^{\overline{n}}(U)\cap T_v^{\overline{m}}(U)\neq \emptyset$. By property \textit{i.)} of $E$, we have that $T_v^m(U)$ is also a connected component of $E$, for every $m\in\Z$. But $E$ is invariant under $\tilde f$, so $\tilde f$ permutes connected components of $E$. Therefore 
$\tilde f^{\overline{n}}(U)=T_v^{\overline{m}}(U)$. By induction, we have that 
$\tilde f^{k\overline{n}}(U)=T_v^{k\overline{m}}(U)$, for every $k\in\N.$\medskip

Since $M_0>\dm(U)$ then, 

\begin{equation*}
\left\|\tilde{f} ^{k\overline{n}}(x)-x-k(0,\overline{m})\right\|\leqslant M_0, \ \ \ \textit {for every} \ \ x\in U \ \textit {and} \   k\in \N,
\end{equation*}  
concluding the proof of the Claim $\ref{m1}$.\medskip

Now consider $D$ an integer translate of $[0,1]^2$ containing $w$.\medskip

We claim that there exist a finite open cover $C=\bigcup_{j=1}^s\left\{U_j\right\}$ for $D$, where every $U_j$ is connected and for some $u_j\in\Z^2$,  $U_j+u_j$ is a connected component of $E$.\medskip

In fact, by property $iii.)$ of $E$, it follows that for every point $y\in D$, there is $u_*\in\Z^2$, such that 
$(y+u_*)\in U^*_y$, for some connected component $U^*_y$ of $E$.\medskip 

Denote $U^\prime=U^*_y-u_*$ and $C^\prime=\bigcup_{y\in D}\left\{U^\prime_y\right\}.$ By properties of $E$, the set $C^\prime$ is an open cover for $D$ such that every $U^\prime + u_*$ is a connected component of $E$. Moreover, $C^\prime$ contains the set $U=U^*_w$. Since $D$ is compact, there is a finite subcover for $D$, denote by 

\begin{align*}
C=\bigcup_{j=1}^s\left\{U_j\right\}, \ \ \text{where}\ \left\{U_j\right\}\subset C^\prime,\ \text{for every}\ j=1,\ldots,s,
\end{align*}
as claimed.
\medskip

Set $U^*_j=U_j+u_j$ for some $u_j\in\Z^2$. So we have $\tilde f^n(U_j)=\tilde f^n(U^*_j)-u_j$ is an integer translate of any connected component of $E$, for every $n\in\Z$, because $\ti f$ permutes connected components of $E$. By property $ii.)$ of $E$, $\dm(\tilde f^n(U_j))< M_0,$ for every 
$n\in\Z$ and $j=1,\ldots, s$.\medskip

We may assume $\left\{U_1\right\}=\left\{U\right\}\subset C$, where $U$ is the open set of Claim $\ref{m1}$ (otherwise consider a finite cover $\left\{U\right\}\cup C$). Since $D$ is connected, without lost of generality, we may list the open sets $U_j$, $j=1,\ldots,s$ of $C$, such that 

\begin{equation}\label{uj}
	U_j\cap\left[\bigcup_{1\leqslant l<j}U_l\right]\neq\emptyset,\ \ \textit{for every $\ j=2,\ldots,s$}.
\end{equation}

\begin{afp}\label{s}
There exists $M>0$ and $p\in\Q^2$ such that, $\left\|\tilde f^k(y)-y-kp\right\|\leqslant M,\ $ for every $y\in D$ and $k\in\N$.
\end{afp}

To proof this claim let us first show by induction that 
\begin{equation}\label{induc}
	\left\|\tilde f^{k\overline{n}}(y)-y-k(0,\overline{m})\right\|\leqslant (2j-1)M_0,
\end{equation}
  for every $y\in U_j,\ $ $j=1,\ldots,s,\ $ and $k\in\N$, where $\overline{n}$ and $\overline{m}$ are like Claim $\ref{m1}$.

\begin{itemize}
	\item[i.] The case $j=1$ is Claim	\ref{m1}.
	\item[ii.] Suppose it holds for every $U_l$, with $1\leqslant l\leqslant j$, and let us see the case $j+1$:\medskip
	
	By (\ref{uj}), we have $U_{j+1}\cap U_l\neq \emptyset$, for some $U_l$ with $1\leqslant l \leqslant j$. Then,
	$\tilde f^{\overline{n}}(U_{j+1})\cap \tilde f^{\overline{n}}(U_{l})\neq \emptyset$. By induction we see that,
  $\tilde f^{k\overline{n}}(U_{j+1})\cap \tilde f^{k\overline{n}}(U_{l})\neq \emptyset$, for every $k\in\N.$ \medskip
	
	Since $\dm(\tilde f^n(U_j))< M_0,$ for every $n\in\Z$ and $j=1,\ldots, s$, then for every $y\in U_{j+1}$, $k\in\N$ and some 
	$x\in U_{j+1} \cap U_l$, we have
\begin{align*}
\left\|\tilde f^{k\overline{n}}(y)-y-k(0,\overline{m})\right\| &\leqslant \left\|\tilde f ^{k\overline{n}}(y)-\tilde f ^{k\overline{n}}(x)\right\|+ \left\|\tilde f ^{k\overline{n}}(x)-x-k(0,\overline{m})\right\| + \left\|x-y\right\|\\
& \leqslant M_0 + (2l-1)M_0 + M_0 \leqslant [2(j+1)-1]M_0.
\end{align*}  
\end{itemize}
Thus equation (\ref{induc}) holds.\medskip

Setting $M_1=(2s-1)M_0$, it follows that $\left\|\tilde f ^{k\overline{n}}(y)-y-k(0,\overline{m})\right\|<M_1$, for every $y\in D$ and $k\in N$.
\medskip

Denote by $p=\frac{1}{\overline{n}}(0,\overline{m})\in \Q^2$. We know that given $k\in \N$, there are $t\: ,r \in \N$, such that 
$k=t\overline{n}+r$, where $0\leqslant r\leqslant \overline{n}$.\medskip

Hence, for every $y\in D$ and $k\in \N$,
\begin{align*}
\left\|\tilde{f} ^{k}(y)-y-kp\right\| & =  \left\|\tilde{f} ^{t\overline{n}+r}(y)-y-(t\overline{n}+r)p\right\|\\
& \leqslant \left\|\tilde{f} ^r(\tilde{f}^{t\overline{n}}(y))-\tilde{f}^{t\overline{n}}(y)\right\| + \left\|\tilde{f} ^{t\overline{n}}(y)-y-t(0,\overline{m})\right\|+	\left\|rp\right\|\\
&\leqslant n_i\ \sup_{z\in \R^2}\left\|\tilde{f}(z)-z\right\|+ M_1 + r\left\|p\right\|=M,
\end{align*} 
completing the proof of the Claim $\ref{s}$.\medskip

Given $z\in\R^2$, let $u\in\Z^2$ be such that $z+u=y\in D$. \medskip

Thus, 

\begin{align*}
\left\|\tilde{f}^{k}(z)-z-kp\right\| = &\left\|\tilde{f} ^{k}(z+u)-(z+u)-kp\right\|= \left\|\tilde{f} ^{k}(y)-y-kp\right\| 
\leqslant M.
\end{align*}  

Therefore, given $z\in \R^2$ and $k\in\N$, there exists $M>0$ and $p\in\Q^2$, such that $$\left\|\tilde{f}^{k}(z)-z-kp\right\| \leqslant M.$$

This implies that $\de(\tilde{f})<\infty.$
\end{proof}
\medskip

Now let us consider the general case, for every $v\in\Z^2_*$.
\medskip

\begin{main}\label{d1}
If $\tilde f$ is $v$-annular, for some $v\in\Z^2_*$, then $\de_{v^{\bot}}(\tilde h)<\infty$ or $\de (\tilde f)<\infty$.
\end{main}

\begin{proof}
By Remark \ref{v}, follows that if $\tilde f$ is $v$-annular then $A\tilde fA^{-1}$ is $(0,1)$-annular. In that way the homeomorphisms 
$A\tilde fA^{-1}$ and $A\tilde hA^{-1}$ satisfy hypothesis of Proposition \ref{d}.
\end{proof}
\bigskip
%


\subsection*{Acknowledgements}
This work is part of the doctoral thesis, under the supervision of Fabio A. Tal. I am thankful to him for valuable conversations, many corrections and suggestions.


\bibliographystyle{plain}
\bibliography{bibliografia}

\end{document}